\documentclass{article}
\usepackage[utf8]{inputenc}
\usepackage[margin=1in]{geometry}
\usepackage{amsmath} 
\usepackage{todonotes}
\usepackage{acronym}
\usepackage{xcolor}
\usepackage{amsmath}
\usepackage{amssymb}
\usepackage{mathtools}
\usepackage{subcaption}
\usepackage{xcolor}
\usepackage{hyperref}
\hypersetup{colorlinks=true, linkcolor=black, citecolor=black}
\usepackage{algorithm}
\usepackage{algpseudocode} 
\usepackage{mathrsfs}
\usepackage{cite}
\usepackage{adjustbox}
\usepackage{float}

\DeclareMathOperator*{\find}{find}
 
\newcommand{\R}{\mathbb{R}} 
\newcommand{\K}{\mathbb{K}} 
 
\newcommand{\N}{\mathbb{N}}

\newcommand{\A}{\mathcal{A}}
\newcommand{\0}{\mathbf{0}}
\newcommand{\1}{\mathbf{1}}
\newcommand{\psd}{\mathbb{S}}

\DeclarePairedDelimiter{\abs}{\lvert}{\rvert}
\DeclarePairedDelimiter{\norm}{\lVert}{\rVert}

\DeclarePairedDelimiterX{\inp}[2]{\langle}{\rangle}{#1, #2}

\usepackage{amsthm}
\newtheorem{thm}{Theorem}[section]

\newtheorem{prop}[thm]{Proposition}

\newtheorem{rem}{Remark} 
 
\newtheorem{prob}{Problem} 

\newcommand{\du}{\Delta u} 
\newcommand{\dy}{\Delta y} 
\newcommand{\dc}{\mathcal{D}}
\newcommand{\pc}{\mathcal{P}}
\newcommand{\bpc}{\bar{\mathcal{P}}}

\newcommand{\alt}{\textrm{alt}}

\usepackage{acronym}
\usepackage{circuitikz}

\usepackage{natbib} 
\setcitestyle{authoryear,open={(},close={)}} 

\title{\Large Superstabilizing Control of Discrete-Time ARX Models\\ under Error in Variables
}

\author{Jared Miller$^1$, Tianyu Dai $^1$, Mario Sznaier$^1$
\thanks{$^1$J. Miller, T. Dai, and M. Sznaier are with the Robust Systems Lab,  ECE Department, Northeastern University, Boston, MA 02115. (e-mails: \{miller.jare, dai.ti\}@northeastern.edu, msznaier@coe.neu.edu).}
\thanks{J. Miller, T. Dai, and M. Sznaier were partially supported by NSF grants  ECCS--1808381 and CNS--2038493, AFOSR grant FA9550-19-1-0005, and ONR grant N00014-21-1-2431. 
J. Miller was in part supported by the Chateaubriand Fellowship of the Office for Science \& Technology of the Embassy of France in the United States, AFOSR grant FA9550-19-1-0005, and the AFOSR International Student Exchange Program.}}

\begin{document}

\maketitle


\begin{abstract}
\label{sec:abstract}


This paper applies a polynomial optimization-based framework towards the superstabilizing control of an Autoregressive with Exogenous Input (ARX) model given noisy data observations. The recorded input and output values are corrupted with L-infinity-bounded noise where the bounds are known. This is an instance of Error in Variables (EIV) in which true internal state of the ARX system remains unknown. The consistency set of ARX models compatible with noisy data has a bilinearity between unknown plant parameters and unknown noise terms. The requirement for a  dynamic compensator to superstabilize all consistent plants is expressed using polynomial nonnegativity constraints, and solved using sum-of-squares (SOS) methods in a converging hierarchy of semidefinite programs in increasing size. The computational complexity of this method may be reduced by applying a Theorem of Alternatives to eliminate the noise terms. The effectiveness of this method is demonstrated on control of example ARX models. 

\end{abstract}
\section{Introduction}
\label{sec:introduction}
\ac{DDC} is a class of algorithms that perform control of plants consistent with measured data without requiring a system identification step first \citep{HOU20133}. Recent design methods for state space models are based on Willem's Fundamental Lemma \citep{depersis2020formulas, berberich2020robust}, Matrix S-Lemma \citep{waarde2020noisy}, and/or a theorem of alternatives \citep{dai2018moments}.

Most \ac{DDC} applications involve structured process noise in dynamics. The \ac{EIV} case is considerably less studied for control. Prior work about system identification and state estimation of \ac{EIV}-affected systems includes \citep{norton1987identification, cerone1993feasible, cerone2011set, soderstrom2018errors}. Polynomial optimization for system identification of \ac{EIV} \ac{ARX} models has previously been considered in Chapter 6 of \citep{cheng2016robust} using rank-minimization, but the \ac{EIV} control task was not previously addressed. This paper is a continuation of research started by the authors in \citep{miller2022eiv_short, miller2022eiv} for \ac{EIV} control with full-state feedback.


Superstability is a computationally tractable restriction of stability used in control and static output feedback \citep{polyak2001optimal, polyak2002superstable}. Superstability allows for the design of controllers with guaranteed performance bounds on the growth of the $L_\infty$-norm of the state/output
\citep{sznaier1995siso}, which was extended to the matrix transfer function case in \citep{chen2005design}.
Data-driven superstabilization of \ac{ARX} models under process noise was performed in \citep{CHENG2015}.

To the best of our knowledge, this paper is the first work that addresses output-feedback \ac{DDC} with noisy input-output measurements.

The contributions of this work are,
\begin{itemize}
    \item Formulation and solution of \ac{EIV} \ac{ARX} superstabilization as \iac{POP} using \ac{SOS} methods.
    \item Application of a Theorem of Alternatives to form a \ac{POP} where noise terms $(\du, \dy)$ are eliminated.
    \item Accounting of the computational complexity to solve the \acp{SDP} with and without utilizing the Theorem of Alternatives.
\end{itemize}

This paper has the following structure: 
Section \ref{sec:preliminaries} introduces preliminaries such as notation, \ac{ARX} models, Superstability applied to \ac{ARX} models, and \ac{SOS} methods. Section \ref{sec:full} presents \iac{POP} for superstabilizing control involving the variables $(a, b, \du, \dy)$. Section \ref{sec:altern} utilizes a Theorem of Alternatives to create an equivalent \ac{POP} in terms of $(a,b)$. Section \ref{sec:complexity} tabulates the computational complexity of these \acp{SDP}. Section \ref{sec:examples} applies the derived methods towards example superstabilization problems. 
Section \ref{sec:extensions} introduces extensions such as process noise and Switched \ac{ARX} models. 
Section \ref{sec:conclusion} concludes the paper.

\section{Preliminaries}
\label{sec:preliminaries}

\begin{acronym}[WSOS]

\acro{ARX}{Autoregressive with Exogenous Inputs}

\acro{BSA}{Basic Semialgebraic}
\acro{DDC}{Data Driven Control}

\acro{EIV}{Error in Variables}
\acroindefinite{EIV}{an}{a}






\acro{LP}{Linear Program}
\acroindefinite{LP}{an}{a}


\acro{POP}{Polynomial Optimization Problem}

\acro{PSD}{Positive Semidefinite}


\acro{SDP}{Semidefinite Program}
\acroindefinite{SDP}{an}{a}


\acro{SOS}{Sum of Squares}
\acroindefinite{SOS}{an}{a}

\acro{WSOS}{Weighted Sum of Squares}

\end{acronym}

\subsection{Notation}
The set of real numbers is $\R$ and the $n$-dimensional real Euclidean vector space is $\R^n$. The set of integers between $r$ and $s$ is $r..s$. The $n$-dimensional nonnegative real orthant is $\R^n_+$. The set of $n \times m$ real-valued matrices is $\R^{n \times m}$. The transpose of a matrix $M$ is $M^T$. The identity matrix is $I$, the all-ones matrix is $\1$, and the all-zeros matrix is $\0$. The $L_1$-norm of a vector $x$ is $\norm{x}_1 = \sum_{i=1}^n \abs{x_i}$. The $L_\infty$ norm of a vector $x$ is $\norm{x}_\infty = \max_{i\in1..n} \abs{x_i}$ The vertical concatenation of vectors $x$ and $y$ is $[x;y]$. 
The lag operator $\lambda$ applied to a sequence $\{x_t\}$ is $\lambda x_t = x_{t-1}$. 
The cross-correlation $x\star y$ between  $\{x_t\}_{t=l}^k$ and $\{y_t\}_{t=m}^n$ is
\begin{equation}
\label{eq:cross_correlation}
    (x\star y)_j = \textstyle \sum_{i=m}^{n} x_{i+j-n+m-1}y_i, \quad j = l..k+n-m.
\end{equation}

The cardinality of a finite set $\A$ is $\abs{\A}.$
The set of natural numbers is $\N$, and the set of $n$-dimensional multi-indices is $\N^n$. 
Every polynomial $p(x)$ may be defined with respect to some set $\A \subset \N^n$ as $p(x) = \sum_{\alpha \in \A} c_\alpha x^\alpha$ with all coefficients $c_\alpha \neq 0$. The degree of a polynomial $p(x)$ is $\deg p$. The ring of real-valued polynomials with indeterminates $x$ is $\R[x]$, and its subset of polynomials with degree $d$ or less is $\R[x]_{\leq d}$. The notation $(\R[x])^k$ will denote a $k$-vector of polynomials in $x$.
The coefficient vector of a polynomial $p(x)$ is $c = \textbf{coeff}[p(x)]$.


\subsection{ARX Models}

An \ac{ARX} model with input-output sequence $(u, y)$ and parameters $(a, b)$ such that $n_a > n_b$ obeys dynamics for times $t \in 1..T$:
\begin{align}
    y_{t} &= -\textstyle \sum_{i=1}^{n_a} a_i y_{t-i} + \sum_{i=1}^{n_b} b_i u_{t-i}.  \label{eq:arx_ground_truth}
\end{align}

The \ac{ARX} model in \eqref{eq:arx_ground_truth} may be represented as a rational transfer function in the lag operator $\lambda$ with
\begin{align}
    G(\lambda) = \frac{\sum_{i=1}^{n_b} b_i \lambda^{i}}{1 + \sum_{i=1}^{n_a} a_i \lambda^{i} } = \frac{B(\lambda)}{1 + A(\lambda)}.\label{eq:transfer_arx}
\end{align}

The numerator $B$ and denominator $A$ of \eqref{eq:transfer_arx} are each polynomials in $\lambda$. The function \eqref{eq:transfer_arx} satisfies $B(0) = 0$ and is assumed to be strictly proper ($G(0) = 0$).

Let $C(\lambda)$ be a dynamic compensator with parameters $(\tilde{a}, \tilde{b})$ of length $(\tilde{n}_a, \tilde{n}_b)$ following the structure of \eqref{eq:transfer_arx},
yielding polynomials $(\tilde{A}(\lambda), \tilde{B}(\lambda))$ with $C(\lambda) = \tilde{B}(\lambda)/(1+\tilde{A}(\lambda))$.
Application of $C(\lambda)$ in feedback with $G(\lambda)$ yields the following closed-loop transfer function $G_{cl}(\lambda)$:
\begin{align}
    \frac{G(\lambda)}{1 + G(\lambda) C(\lambda)} = \frac{B(\lambda) (1+\tilde{A}(\lambda))}{(1+A(\lambda))(1+ \tilde{A}(\lambda)) + B(\lambda) \tilde{B}(\lambda)}. \label{eq:closed_loop}
\end{align}

The closed-loop system parameter $a^{cl}$ related to \eqref{eq:closed_loop} is
\begin{align}
\label{eq:closed_loop_coeff}
a^{cl} &= \textbf{coeff}[(1+A(\lambda)) (1+\tilde{A}(\lambda)) + B(\lambda) \tilde{B}(\lambda) - 1].
\end{align}

A single \ac{ARX} system in \eqref{eq:arx_ground_truth} is stable if the roots of $1+ A(\lambda)$ are contained within the exterior of the unit disc $\{\lambda \mid \sqrt{\textrm{Re}(\lambda)^2 + \textrm{Im}(\lambda)^2} > 1\}$. This stability may be verified through numerical computation of roots  or with the Jury stability criterion \citep{ogata1995discrete}. These methods are computationally expensive to employ when designing stabilizing compensators $C$ (as in \eqref{eq:closed_loop})such that \eqref{eq:closed_loop_coeff} is stable, given that the methods all involve polynomial constraints on the entries of $(\tilde{a}, \tilde{b})$.

\subsection{Superstability}
Superstability \citep{blanchini1995persistent, polyak2001optimal, polyak2002superstable} is a conservative notion of stability that possesses simpler computational properties. An \ac{ARX} system is superstable if $\norm{a}_1 < 1$.
If the system satisfies $\norm{a}_1 < \gamma$ for some value $\gamma \in [0, 1)$, then a bound on  $\abs{y_t}$ based on the initial condition given $u = 0$ may be found as $\forall t\geq 0: \ \abs{y_t} \leq \gamma^{t/n_a + 1}\max_{t'\in 0..n_a-1} \abs{x_{-t'}}$ (Theorem A.1.(a) of \citep{polyak2002superstable} when starting at time $t=1$).
The superstability norm constraint $\norm{a^{cl}}_1 < 1$ from  \eqref{eq:closed_loop_coeff} is affine-expressible in the compensator parameters $(\tilde{a}, \tilde{b})$. 


Superstability may be imposed for classes of plants. A single controller $C$ can simultaneously superstabilize a set $(a, b)$ of plants if the closed-loop system \eqref{eq:closed_loop} is superstable for each individual plant.






\subsection{Sum of Squares}

\Iac{BSA} set $\K$ is the locus of a finite number of bounded-degree polynomial inequality and equality constraints:
\begin{align}
    \K = \left\{x \in \R^n \mid g_k(x) \geq 0, h_{k'}(x) = 0\right\},
\end{align}
for all indices $k=1..N_i, \ k' = 1..N_e$. \ac{BSA} sets may be intersected by concatenating their describing polynomials $\{g\}$ and $\{h\}$. 
The $x$-projection operator $\pi^x: X \times Y \rightarrow X$ is $\pi^x: (x, y) \mapsto x$.
\Iac{BSA} set $\K \subset X \times Y$ has an $x$-projection $\pi^x \K = \{x \in X \mid (x, y) \in \K\}$.
\ac{BSA} sets are not closed under projection, instead the projection of a \ac{BSA} set is generically the union of disjoint \ac{BSA} sets. 

A polynomial is nonnegative if $\forall x \in \R^n: \ p(x) \geq 0$. A polynomial $p(x)$ is \ac{SOS} if there exists an integer $s$, a polynomial vector $v \in (\R[x])^s$, and an $s \times s$ \ac{PSD} \textit{Gram} matrix $Q \succeq 0$ such that $p(x) = v(x)^T Q v(x)$. \ac{SOS} polynomials are a subset of nonnegative polynomials, given that the square of any real number is nonnegative. The set of \ac{SOS} polynomials is $\Sigma[x] \subset \R[x]$, and its subset of degree-$2d$ polynomials is $\Sigma[x]_{\leq 2d} \subset \R[x]_{\leq 2d}$ (\ac{SOS} polynomials are always even).

A sufficient condition for a polynomial $p(x)$ to be positive over a \ac{BSA} set $\mathbb{K}$ is \citep{putinar1993compact}
\begin{subequations}
\label{eq:putinar}
    \begin{align}
        & p(x) = \sigma_0(x) + \textstyle \sum_i {\sigma_i(x)g_i(x)} + \textstyle \sum_j {\phi_j(x) h_j}\\
        &\exists  \sigma_0(x) \in \Sigma[x], \quad \sigma(x) \in (\Sigma[x])^{N_g}, \quad \phi \in (\R[x])^{N_h}. \label{eq:putinar_variables}
    \end{align}
\end{subequations}
The \ac{WSOS} set $\Sigma[\K]$ is the set of polynomials that admit a certificate in \eqref{eq:putinar} (called a Putinar Psatz). The Putinar certificate in \eqref{eq:putinar} is necessary and sufficient if an \textit{Archimedean} condition holds: $\exists R > 0 \mid \ R-\norm{x}_2^2 \in \Sigma[\K]$. Every Archimedean \ac{BSA} set is compact, and compact sets may be rendered Archimedean by adjoining $R-\norm{x}^2_2 \geq 0$ to the descriptor constraints if a valid $R$ is previously known. The moment-\ac{SOS} hierarchy for optimization involves increasing the degree $2d$ to obtain higher-order Putinar multipliers \eqref{eq:putinar_variables} \citep{lasserre2009moments}. 
Certifying that a degree-$2d$ polynomial $p(x)$ is \ac{SOS} requires $1+N_g$ Gram matrices of maximal size $\binom{n+d}{d}$, and the approximate per-iteration runtime of an \ac{SDP} originating from the moment-\ac{SOS} hierarchy is approximately $O(n^{6d})$ when $d$ is fixed. 




\section{Superstabilizing Control}
\label{sec:full}
This section presents \iac{POP} to perform superstabilizing control of \ac{EIV} \ac{ARX} models.

\subsection{Problem Statement}
A set of input-output observations $\dc = (\hat{u}, \hat{y})$ are recorded for an \ac{ARX} dynamical system within a time horizon of $T$. The ground-truth is the input-output data $(u, y)$ following dynamics \eqref{eq:arx_ground_truth}. 
The records in $\dc$ are corrupted by $L_\infty$-bounded input noise $(\du \in \R^{T+n_b-1})$ and measurement noise  $(\dy \in \R^{T+n_a})$ as
\begin{subequations}
\label{eq:dydu_corruption}
\begin{align}
    \hat{y} &= y + \dy, & &\norm{\dy}_\infty \leq \epsilon_y \\
    \hat{u} &= u + \du, & &\norm{\du}_\infty \leq \epsilon_u.
\end{align}
\end{subequations}
The combination of input and measurement noise is the \ac{EIV} setting.
Substitution of \eqref{eq:dydu_corruption} into \eqref{eq:arx_ground_truth} yields
\begin{align}
    \hat{y}_{t} - \dy_t &= \left(-\textstyle \sum_{i=1}^{n_a} a_i \hat{y}_{t-i} + \sum_{i=1}^{n_b} b_i \hat{u}_{t-i}\right) \nonumber \\
    &- \left(-\textstyle \sum_{i=1}^{n_a} a_i \dy_{t-i} + \sum_{i=1}^{n_b} b_i \du_{t-i}\right).  \label{eq:arx_bilinear}
\end{align}

\subsection{Consistency Sets}

The $(\du,\dy)$-constant terms in \eqref{eq:arx_bilinear} may be written as
\begin{align}
\label{eq:dudy_constant}
    h_t = \hat{y}_t + \textstyle \sum_{i=1}^{n_a} a_i \hat{y}_{t-i} - \sum_{i=1}^{n_b} b_i \hat{u}_{t-i}.
\end{align}

Eq. \eqref{eq:arx_bilinear} may therefore be expressed  in terms of $h_t$ as
\begin{align}
    0 = h_t + \left(-\textstyle \sum_{i=1}^{n_a} a_i \dy_{t-i} + \sum_{i=1}^{n_b} b_i \du_{t-i}\right) - \dy_t. \label{eq:arx_bilinear_h}
\end{align}

The set of parameters $(a, b)$ and noise values $(\du, \dy)$ consistent with data $\dc$ is
\begin{equation}
\label{eq:bpc}
    \bpc: \left\{\begin{array}{c|c}
    a \in \R^{n_a}, \ b \in \R^{n_b} & \norm{\du}_\infty \leq \epsilon_u  \\
    \du \in \R^{T+n_b-1}  & \norm{\dy}_\infty \leq \epsilon_y \\
    \dy \in \R^{T+n_a} & \textrm{Eq.  \eqref{eq:arx_bilinear_h}} \ \forall t = 1..T \\
    \end{array}\right\}.
\end{equation}

The set of parameters $(a, b)$ consistent with data in $\dc$ is
\begin{equation}
    \label{eq:pc}
    \pc(a, b)  = \pi^{a,b} \bpc(a,b,\du, \dy).
\end{equation}
Equivalently, a plant $(a, b)$ is a member of $\pc$ if there exists an admissible $L_\infty$-bounded noise process $(\du, \dy)$ that could have generated $\dc$.

The following assumption will be required to obtain convergence:
\begin{itemize}
    \item[A1] The set $\bpc$ is compact.
\end{itemize}

\begin{rem}
Compactness of $\bpc$ by A1 implies that its projection $\pc$ is also compact.
\end{rem}

\begin{rem}
Determining membership (if a fixed plant $(a, b) \in \pc$) is  \iac{LP} with variables $\du, \dy$, given that $\bpc$ has a bilinear description.
\end{rem}

\subsection{Full Program}

The coefficients $a^{cl}$ have maximal length $n_{cl} = \tilde{n}_a + n_a$.
\begin{prob}
\label{prob:full_function}
A program to perform superstabilizing control (or to find an infeasibility certificate if $\gamma^* \geq 1$) is
\begin{subequations}
\label{eq:full_function}
\begin{align}
    \gamma^* =& \min_{\gamma \in \R, (\tilde{a}, \tilde{b})}  \gamma   &  \forall (a, b, \du, \dy) \in \bpc: \  \gamma \geq \norm{a^{cl}}_1. & 
\end{align}
\end{subequations}
\end{prob}

The $L_1$-norm constraint in \eqref{eq:full_function} may be equivalently represented by a lifted sequence of inequalities. The $L_1$-norm of a vector $x\in \R^n$ may also be expressed as $\norm{x}_1 = \min_{m \in \R^{n_{cl}}} \sum_i m_i: \ -m_i \leq x_i \leq m_i$ \citep{gouveia2013lifts}. Letting $m_i(a, b): \pc \rightarrow \R$ be a set of functions for $i=1..n_a + \tilde{n}_a$, expression \eqref{eq:full_function} may be written as
\begin{subequations}
\label{eq:full_function_m}
\begin{align}
    &\gamma^* =\min_{\gamma, (\tilde{a}, \tilde{b})}  \gamma \\
    & \quad \forall (a, b, \du, \dy) \in \bpc: \\
    & \qquad \gamma - \textstyle \sum_{i=1}^{n_{cl}} m_{i}(a,b)\geq 0 \label{eq:m_gamma}\\
    & \qquad m_{i}(a,b) - a^{cl}_i(a,b)\geq 0 & & \forall i=1..n_{cl} \label{eq:m_pos}\\
    & \qquad m_{i}(a,b) + a^{cl}_i(a,b) \geq 0 & & \forall i=1..n_{cl}.\label{eq:m_neg}
\end{align}
\end{subequations}


\subsection{Full Sum-of-Squares Program}

Problem \ref{prob:full_function} may be solved using \ac{SOS} programming. The nonnegativity constraints in \eqref{eq:m_gamma}-\eqref{eq:m_neg} may each be realized as Psatz constraints in the sense of \eqref{eq:putinar}. 
The degree-$d$ polynomial restriction to Problem \ref{prob:full_function} is presented in  Eq. \eqref{eq:super_full_wsos}. The decision variables of  Eq. \eqref{eq:super_full_wsos} are the $L_1$-certificates $m$ and the gain $\gamma$.
The $\geq$ symbols in \eqref{eq:full_function_m} are tightened to $>$ in Eq. \eqref{eq:super_full_wsos} due to  the Putinar Psatz's \eqref{eq:putinar} positivity certificate. 
\begin{subequations}
\label{eq:super_full_wsos}
\begin{align}
\gamma_d^*&= \min_{\gamma \in \R} \gamma \\    
& \tilde{a}\in \R^{n_a}, \ \tilde{b} \in \R^{n_b}\\
    &m_i \in (\R[a, b])^{n_{cl}}_{\leq 2d}  & & \forall i \in 1..n_{cl}\\    
    & \gamma - \textstyle \sum_{i=1}^{n_{cl}} m_i(a, b) \in \Sigma[\bpc]_{\leq 2d} \label{eq:super_full_wsos_put_gamma}\\
    &  m_{i}(a,b) - a^{cl}_i(a,b)\in \Sigma[\bpc]_{\leq 2d} & & \forall i=1..n_{cl} \label{eq:super_full_wsos_put_p}\\
    &  m_{i}(a,b) + a^{cl}_i(a,b)\in \Sigma[\bpc]_{\leq 2d}& &\forall i=1..n_{cl}.\label{eq:super_full_wsos_put_m}
\end{align}
\end{subequations}
\begin{thm}
\label{thm:m_cont_selection}
The functions $m_i$ have continuous selections.
\end{thm}
\begin{proof}
 Let $\mathcal{M}_{\gamma, C}:  \pc \rightrightarrows \R^{n_{cl}}$ be the set-valued map defining the closed convex solution region $(a, b) \mapsto \{m \in \R^{n_{cl}} \mid \sum_{i=1}^{n_{cl}} m_i \leq \gamma, \ \forall i: m_i \geq \pm a_i^{cl}(a, b) \}$ from \eqref{eq:full_function_m} given $(\gamma, \tilde{a}, \tilde{b})$. Note that the functions $a^{cl}_i(a, b)$ from \eqref{eq:closed_loop_coeff} are linear (continuous) functions of $(a, b)$ given $(\tilde{a}, \tilde{b})$. Theorem 2.4 of \citep{mangasarian1987lipschitz} proves that $\mathcal{M}$ is a lower-semicontinous map (image of linear inequalities under perturbations in the right-hand side). Michael's Theorem (9.1.2 in \citep{aubin2009set}) suffices to show that a continuous selection exists, because $\mathcal{M}_{\gamma, C}$ is lower-semicontinuous with closed convex images, $\pc$ is compact, and $\R^{n_{cl}}$ is a Banach space. 
\end{proof}

\begin{thm}
\label{thm:m_poly}
The bounds  from Eq. \eqref{eq:super_full_wsos} will converge to $\lim_{d \rightarrow \infty} \gamma^*_d = \gamma^*$ from Problem \ref{prob:full_function} with $\gamma_d^* \geq \gamma_{d+1}^* \geq \ldots \gamma^*$ when $\bpc$ is Archimedean.
\end{thm}
\begin{proof}
Let $(\gamma, m)$ be a solution to Problem \ref{prob:full_function} with each $m_i(a, b)$ continuous over $\pc$ (by Theorem \ref{thm:m_cont_selection}). For each $\epsilon >0$, there exists functions $\tilde{m_i}(a, b) \in \R[a,b]$ such that $\sup_{(a,b) \in \pc} \abs{(m_i(a,b) +\epsilon)- \tilde{m}(a,b)} \leq \epsilon$ by the Stone Weierstrass theorem over the compact $\pc$. Define $r_i(a, b) = m_i(a,b) + \epsilon - \tilde{m}_i(a,b)$ as the approximation error.
The finite-degree polynomial variates $(\gamma + n_{cl} \epsilon, \tilde{m_i})$ are feasible solutions to \eqref{eq:full_function} with
\begin{subequations}
\begin{align}
0 \leq &m_i(a, b) - a_{cl}(a, b) \leq \tilde{m}_i(a,b) - a_{cl}(a,b) \label{eq:sw_m}\\
0 \leq & \gamma - \textstyle \sum_{i} m_i(a,b)  \leq  (\gamma + n_{cl}) \epsilon - \textstyle \sum_{i} \tilde{m}_i(a,b). \label{eq:sw_gamma}
\end{align}
\end{subequations}

The multipliers for Putinar Psatz certificates for \eqref{eq:sw_m} and \eqref{eq:sw_gamma} have finite (exponential) degree in terms of $\deg \tilde{m_i}$ \citep{nie2007complexity}. Therefore it holds that for each $\epsilon$, there exists some degree $d$ such that $\gamma_d = \gamma + n_{cl} \epsilon$. Because $\lim_{\epsilon \rightarrow 0} \gamma^* + n_{cl} \epsilon = \gamma^*$, it holds that $\lim_{d \rightarrow \infty} \gamma_d^* = \gamma^*$. The sequence is decreasing with $\gamma_d^* \geq \gamma_{d+1}^*$ because the \ac{WSOS} cones satisfy $\Sigma[\bpc]_{\leq 2d} \subset \Sigma[\bpc]_{\leq 2(d+1)}$.

\end{proof}


\section{Alternatives Control}
\label{sec:altern}

 Eq. \eqref{eq:super_full_wsos} involves a total of $2(n_a + n_b + T)-1$ variables $(a, b, \du, \dy)$. The Psatz expressions such as in \eqref{eq:super_full_wsos_put_gamma} will therefore have Gram matrices of size $\binom{2(n_a + n_b + T) + d}{d}$ at each fixed degree $d$. Performance of  Eq. \eqref{eq:super_full_wsos} is therefore polynomial in $T$ for fixed $(n_a,n_b, d)$, and is jointly combinatorial in all parameters. 

A theorem of alternatives may be utilized to eliminate the noise terms $(\du, \dy)$ from Putinar expressions. This Alternatives algorithm scales in a linear manner based on $T$ and possesses Gram matrices of maximal size  $\binom{n_a+n_b + d}{d}$. 
Letting $n_a = 3, n_b = 2, T = 10$, this maximal size is $\binom{15}{2}=465$ for Full and $\binom{5+1}{1} = 6$ for Alternatives.



\subsection{ARX Alternatives Psatz}

Let $q(a, b)$ be a function defined over $\pc$. Because $q(a,b)$ is a function of $(a,b)$ alone, the following positivity criteria are equivalent by projection:
\begin{subequations}
\label{eq:q_forall}
\begin{align}
    q(a,b)&> 0 & &\forall (a,b) \in \pc \\
    q(a,b)&> 0 & &\forall (a,b,\du,\dy) \in \bpc.
\end{align}
\end{subequations}
The following statement is a strong alternative to \eqref{eq:q_forall}:
\begin{align}
\label{eq:q_exists}
    \exists(a,b, \du, \dy) & \in \bpc: & -q(a, b) \geq 0.
\end{align}
Dual variable functions may be defined according to the Putinar multipliers in constraint description \eqref{eq:bpc} with
\begin{subequations}
\label{eq:dual_multipliers}
\begin{align}
    &\psi_t^\pm(a, b) \geq 0 & & \forall t=(-n_b+1)..T-1 \label{eq:dual_multipliers_psi}\\
    &\zeta_t^\pm(a, b) \geq 0 & & \forall t=(-n_a+1)..T \label{eq:dual_multipliers_lambda}\\
    &\mu_t(a, b)  & & \forall t=1..T.
\end{align}
\end{subequations}
The $(a,b)$ dependence in terms $(\psi^\pm, \zeta^\pm, \mu)$ will be omitted to simplify notation. Additionally, the term $\psi^+$ will refer to the vector $\{\psi^+_t\}_{t=1}^T$ (with similar vectorial definitions for $\psi^-, \zeta^\pm, \mu$).
The nonnegativity constraints in \eqref{eq:dual_multipliers_psi} and \eqref{eq:dual_multipliers_lambda} are required to hold for all $(a,b) \in \pc$.

A weighted sum $\Phi$ may be developed from $q$,  multipliers in \eqref{eq:dual_multipliers}, and the description \eqref{eq:bpc}, by forming
\begin{align}
    \Phi &= -q(a, b) +\textstyle \sum_{t=1}^{T} \mu_t h_t \label{eq:phi_all}  \\
    &+ \textstyle \sum_{t=1}^{T} \mu_t (\sum_{i=1}^{n_b} b_i \du_{t-i} -\textstyle \sum_{i=1}^{n_a} a_i \dy_{t-i}  - \dy_t) \nonumber\\
    &+ \textstyle \sum_{t=-(n_b-1)}^{T-1} \left(\psi^+_{t}(\epsilon_u - \du_t) +\psi^-_{t}(\epsilon_u + \du_t)\right) \nonumber\\
    &+\textstyle \sum_{t=-(n_a-1)}^{T} \left(\zeta_{t}^+(\epsilon_y - \dy_t) +\zeta^-_{t}(\epsilon_y + \dy_t)\right). \nonumber\\
\intertext{The terms in \eqref{eq:phi_all} that are constant in $(\du, \dy)$ may be collected into}
    Q(a,b) &= -q(a,b) + \epsilon_u \1^T (\psi^+ + \psi^-) \label{eq:Q_const}\\ 
    &+ h^T \mu + \epsilon_y \1^T(\zeta^+ + \zeta^-). \nonumber
    \end{align}
Using the cross-correlation operator $\star$ from \eqref{eq:cross_correlation}, Eq. \eqref{eq:phi_all} may be rewritten as
    \begin{align}
 \Phi &= Q(a,b) + (\mu \star b)^T \du - (\mu \star [1;a])^T \dy \label{eq:phi_collect}\\
 &+ (\psi^- - \psi^+)^T \du + (\zeta^- - \zeta^+)^T \dy. \nonumber
\end{align}
\begin{align}
\intertext{A Lagrangian dual function $g(a,b)$ may be defined as}
    g(a,b) &= \sup_{\du \in \R^{T+n_b}, \ \dy \in \R^{T+n_a}} \Phi(a,b,\du,\dy).
    \intertext{The coefficients of $(\du, \dy)$ in \eqref{eq:phi_collect} must be zero in order to ensure that $g(a, b)$ is bounded. The value of this dual function is}
    g(a,b) &= \begin{cases}Q(a, b) & \psi^+ - \psi^- = 
    \mu \star b \\
    & \zeta^+ - \zeta^- 
    = \mu \star [1;a] \\
    \infty & \textrm{Else}. \end{cases}
\end{align}

\begin{prob}
\label{prob:altern_feas}
A feasibility program to certify \eqref{eq:q_forall} is
\begin{subequations}
\label{eq:altern_feas}
\begin{align}
    \find_{\psi^\pm, \ \zeta^\pm, \ \mu \ \textrm{from \eqref{eq:dual_multipliers}}} & -Q(a,b) > 0 \quad \forall (a,b) \in \pc\\ 
    & \psi^+ - \psi^- = 
    \mu \star b \\
    & \zeta^+ - \zeta^- 
    = \mu \star [1;a]. 
\end{align}
\end{subequations}
\end{prob}

\begin{thm}
Problem \ref{prob:altern_feas} certifies  \eqref{eq:q_forall} and is a strong alternative of \eqref{eq:q_exists}.
\end{thm}
\begin{proof}
\textit{Sufficiency:} The statement $g(a,b) < 0$ is a sufficient condition for invalidation of \eqref{eq:q_exists}. Assuming that \eqref{eq:q_exists} holds with $-q \geq 0$, then $\Phi$ is constructed by the addition of the nonnegative $-q$ plus nonnegative weights $(\phi^\pm, \zeta^\pm)$ times nonnegative constraints $(\epsilon_u, \epsilon_y)$ plus free weights $\mu_t$ times data-consistency equality conditions \eqref{eq:arx_bilinear_h}. The supremal value of a nonnegative term $\Phi$ being negative with $g(a,b) <0$ is a contradiction.

\textit{Necessity:} The constraints in \eqref{eq:bpc} are affine  $(\du, \dy)$. Necessity follows if constraints are concave (including affine) in the eliminated variables by Section 5.8 \citep{boyd2004convex}.
\end{proof}

\begin{prop}
\label{prop:cont}
The multiplier functions $(\psi^\pm, \ \zeta^\pm, \ \mu)$ may each be chosen to be polynomial in the compact set $\pc$.
\end{prop}
\begin{proof}
This proof is omitted for brevity, as it  follows using arguments from the proof of Theorems 4.4 (continuity) and 4.5 (polynomial approximability) of \citep{miller2022eiv}.
\end{proof}

Forming a Psatz with Problem \ref{prob:altern_feas} requires an  additional assumption:
\begin{itemize}
    \item[A2] An Archimedean set $\Pi \supseteq \pc$ is known in advance. 
\end{itemize}

Eq. \eqref{eq:altern_sos} is a Psatz that can certify \eqref{eq:q_forall} at degree $d$. 
\begin{subequations}
\label{eq:altern_sos}
\begin{align}
     &\psi^\pm(a,b) \in (\Sigma[\Pi]_{\leq 2d})^{T+n_b-1} & &   \label{eq:altern_sos_psi} \\
     &  \zeta^\pm(a,b) \in (\Sigma[\Pi]_{\leq 2d})^{T+n_a} & &    \label{eq:altern_sos_lambda} \\
    & \mu(a,b) \in (\R[a,b]_{\leq 2d-1})^{T} & &  \label{eq:altern_sos_mu}\\
    & -Q(a,b; \ \psi^\pm, \zeta^\pm; \mu) \in \Sigma[\Pi]_{\leq 2 d} \textrm{ (from \eqref{eq:Q_const})}\label{eq:altern_sos_Q}\\
    & \psi^+ - \psi^- = 
    \mu \star b, \quad \zeta^+ - \zeta^- = 
    \mu \star [1;a].
    \label{eq:altern_eq_psi}    
\end{align}
\end{subequations}

\begin{prop}
\label{prop:poly}
Eq. \eqref{eq:altern_sos} will converge to a positivity certificate for \eqref{eq:q_forall} as $d\rightarrow \infty$ under A1 and A2.
\end{prop}
\begin{proof}
Proposition \ref{prop:cont} ensures that there exists at least one polynomial certificate $(\psi^\pm, \ \zeta^\pm, \ \mu)$. Given that polynomials have finite degree, letting $d \rightarrow \infty$ will ensure that a polynomial will be reached at some finite $d$.
\end{proof}


\subsection{Alternatives for Superstabilization}

The Alternatives Psatz in Eq. \eqref{eq:altern_sos} may be employed for superstabilizing \ac{EIV} control. 
The Alternatives program to solve Problem \ref{prob:full_function} is described in Eq. \eqref{eq:super_altern_wsos}. Eq. \eqref{eq:super_altern_wsos} is structurally identical to \eqref{eq:super_full_wsos}, with the difference that $m$ is now a function of $(a,b)$ and the positivity constraints in \eqref{eq:super_alt_wsos_put_gamma}-\eqref{eq:super_alt_wsos_put_m} are imposed using Eq. \eqref{eq:altern_sos} $(\Sigma[\pc]^\alt_{\leq2d})$.
\begin{subequations}
\label{eq:super_altern_wsos}
\begin{align}
\gamma_d^*&= \min_{\gamma \in \R} \gamma \\
    &\tilde{a} \in \R^{n_a}, \ \tilde{b} \in \R^{n_b} \\
    & m_i \in (\R[a, b])^{n_{cl}}_{\leq 2d}  & & \forall i \in 1..n_{cl}\\
    & \gamma - \textstyle \sum_{i=1}^{n_{cl}} m_i(a, b) \in \Sigma[\bpc]_{\leq 2d} \label{eq:super_alt_wsos_put_gamma}\\
    &  m_{i}(a,b) - a^{cl}_i(a,b)\in \Sigma[\bpc]_{\leq 2d} & & \forall i=1..n_{cl} \label{eq:super_alt_wsos_put_p}\\
    &  m_{i}(a,b) + a^{cl}_i(a,b)\in \Sigma[\bpc]_{\leq 2d}& &\forall i=1..n_{cl}.\label{eq:super_alt_wsos_put_m}
\end{align}
\end{subequations}
\section{Computational Complexity}
\label{sec:complexity}

This section will tabulate the computational complexity of imposing that a polynomial $q(a,b) \in \R[a,b]_{\leq 2d}$ is nonnegative over $\pc$ using the Putinar (Full) Psatz in \eqref{eq:putinar} and the Alternatives Psatz in \eqref{eq:altern_sos}.
We will use the quantity $N = n_a + n_b$ in this analysis.
The Full program in  \eqref{eq:super_full_wsos} and the Alternatives program in Eq. \eqref{eq:super_altern_wsos} each have $2 n_{cl} + 1$ instances of their respective Psatz certificates. 
Table \ref{tab:ss} compiles the sizes of the optimization variables in the Full and Alternatives programs. The notation $\R$ and $\psd_{+}$ in the table refers to the length of a real vector and the dimension of \iac{PSD} matrix, respectively. This analysis treats the set $\Pi = \R^{n_a + n_b}$ in the Alternatives program to simplify tabulation. If the \ac{BSA} set $\Pi$ has $N_\Pi$ polynomial-defined constraints, then the Alternatives program has 1 set of variables corresponding to entries in the Alternatives column in \ref{tab:ss} and $N_\Pi$ sets of variable with smaller sizes. 
\begin{table}[h]
    \centering
        \caption{Size of Superstabilizing Psatz.}
    \label{tab:ss}
\begin{tabular}{l|l|l|l}
              & $\#$ polynomials & Full                           & Alternatives\\ \cline{1-4} 
$\sigma_0$    & 1           & $\psd_+{\binom{2(N+T)-1 + d}{d}}$ &  $\psd_+{\binom{N + d}{d}}$   \\
$\psi^\pm$    & $2(n_b+T-1)$  & $\psd_+{\binom{2(N+T)-1 + d-1}{d-1}}$ & $\psd_+{\binom{N + d}{d}}$ \\ 
$\zeta^\pm$ & $2(n_a+T)$  & $\psd_+{\binom{2(N+T)-1 + d-1}{d-1}}$ & $\psd_+{\binom{N + d}{d}}$ \\ 
$\mu$         & $T$       & $\R{\binom{2(N+T)-1 + 2d-2}{2d-2}}$ & $\R{\binom{N + 2d-1}{2d-1}}$ 
    \end{tabular}
\end{table}

\begin{rem}
The Alternatives program is more efficient than the Full program for each $d$ given that $N < 2(N+T)$.
\end{rem}
\begin{rem}
Section \ref{sec:altern} eliminated $(\du, \dy)$ and presented a Psatz \eqref{eq:altern_sos} in terms of the $(n_a + n_b)$ variables $(a, b)$. Given that $\bpc$ in \eqref{eq:bpc} is bilinear in terms of the groups $[(a, b), (\du, \dy)]$ and each $a^{cl}$ in \eqref{eq:closed_loop_coeff} is linear in $(a, b)$, an Alternatives program in terms of $(\du, \dy)$ could have been created by eliminating $(a, b)$. This approach would be more complex than Eq. \eqref{eq:super_altern_wsos}, because $(\du, \dy)$ has a total of $2T + n_a + n_b - 1 > n_a+n_b$ variables.
\end{rem}

\section{Numerical Examples}

\label{sec:examples}

MATLAB (2020b) code to reproduce the below experiments is located at \url{https://github.com/jarmill/eiv_arx}. These routines require Mosek \citep{mosek92} and YALMIP  \citep{lofberg2004yalmip}.

We tested the effectiveness of the proposed method using a discrete-time model of
\begin{equation}\label{eq:sys1}
    G(\lambda) = \frac{\lambda^2}{1+0.5\lambda-1.21\lambda^2-0.605\lambda^3}.
\end{equation}
This system is open-loop unstable with unstable poles $z=\frac{1}{\lambda} = \pm1.1$. For comparison purpose, we first solve the model-based superstabilization problem. This is addressed by minimizing $\gamma$ with $||a^{cl}||_1 \leq \gamma$ and known $a,b$ from \eqref{eq:sys1} to search for control coefficients $\tilde{a},\tilde{b}$. 
A superstabilizing controller will occur with $\gamma < 1$, and a smaller $\gamma$ results in a faster controller.
A special case $\gamma = 0$ corresponds to the deadbeat control, i.e. all closed-loop poles are located at the origin. For model \eqref{eq:sys1}, two types of controllers can be found. The low-order controller is obtained with $n_{\tilde{a}} = 3, n_{\tilde{b}} = 2$ and $\gamma = 0.4417$:
\begin{equation}\label{eq:ctrl1}
C(\lambda)=\frac{1.829\lambda^2}{1-0.5\lambda+1.46\lambda^2-0.73\lambda^3}.
\end{equation}
The deadbeat controller is obtained with $n_{\tilde{a}} = 4, n_{\tilde{b}} = 3$ and $\gamma = 0$:
\begin{equation}\label{eq:ctrl2}
C(\lambda)=\frac{0.73\lambda-1.464\lambda^2-0.8833\lambda^3}{1-0.5\lambda+1.46\lambda^2}
\end{equation}
\begin{rem}
Note that we design using $\lambda = \frac{1}{z}$, the controller is improper in $\lambda$ but is physically realizable in $z$. 
\end{rem}
For the data-driven setup, we used the deadbeat controller as the benchmark, i.e. $\gamma = 0$. The system is excited using uniformly distributed input $||u_t||_\infty\leq 1$, initial output $\{y_t\}_{t=-na+1}^{-1}$ and noise $||\Delta y||_\infty \leq \epsilon_y, ||\Delta u||_\infty \leq \epsilon_u$ with $T=10$ samples.
We start from the noise-free data, i.e. $\epsilon_y = \epsilon_u = 0$. Directly applying Full  Eq. \eqref{eq:super_full_wsos} (with $d = 2$) leads to an intractable problem. Using the Alternatives Eq. \eqref{eq:super_altern_wsos} (with $d = 1$) leads to the same $\ell_1$-optimal deadbeat controller as in \eqref{eq:ctrl2}. 
This indicates that there is no gap between the original problem and its alternative form. 
For specific complexity, see Table \ref{tab:ss_eg}. 
\begin{table}[h]
    \centering
        \caption{Size of Superstabilizing Psatz (Model \eqref{eq:sys1} with $n_a = 3, n_b = 2, T = 10$). }
    \label{tab:ss_eg}
\begin{tabular}{l|l|l|l}
              & $\#$ polynomials & Full                           & Altern.\\ \cline{1-4} 
$\sigma_0$    & 1           & 465 &  6   \\
$\psi^\pm$    & 22 & 30 & 6 \\ 
$\zeta^\pm$ & 26  & 30 & 6 \\ 
$\mu$         & 10       & 465 & 6
    \end{tabular}
\end{table}

Next, we show the effect of changing $T$ (Table \ref{tab:T}) and $\epsilon$ (Table \ref{tab:eps}) on the performance index $\gamma$. 
\begin{table}[h]
    \centering
    \caption{$\gamma$ v.s. $T$ with $\epsilon_y=\epsilon_u = 0.02$. }
    \label{tab:T}
    \begin{adjustbox}{width=0.35\textwidth}
    \begin{tabular}{|l|l|l|l|l|}
    \hline
         $T$ & 20 & 40 & 60 & 80  \\ \hline
         $\gamma$ & 0.4365 & 0.3132 & 0.2732 & 0.2515 \\ \hline
    \end{tabular}
    \end{adjustbox}
\end{table}
\begin{table}[h]
    \centering
    \caption{$\gamma$ v.s. $\epsilon=\epsilon_y = \epsilon_u$ with $T = 80$. }
    \label{tab:eps}
    \begin{adjustbox}{width=0.35\textwidth}
    \begin{tabular}{|l|l|l|l|l|}
    \hline
         $\epsilon$ & 0.02 & 0.04 & 0.06 & 0.08  \\ \hline
         $\gamma$ & 0.2515 & 0.4924 & 0.7312 & 0.9755 \\ \hline
    \end{tabular}
    \end{adjustbox}
\end{table}

Several conclusions can be drawn from the tables. 
\begin{enumerate}
    \item[a)] For the noisy trajectory, $\gamma \neq 0,$  the learned controller is no longer deadbeat. 
    \item[b)] As we increase $T$, the consistency set shrinks, hence we get a faster controller. 
    \item[c)] Larger noise leads to a larger consistency set, which makes it harder to find a single robust controller. This issue may be alleviated by collecting more samples. 
\end{enumerate}


The performance $\gamma$ can also be improved by selecting a higher order controller. This is illustrated by Table \ref{tab:high_order} in which $n_b = n_a -1$.
\begin{table}[h]
    \centering
    \caption{$\gamma$ vs. $n_a$ with $\epsilon_y = \epsilon_u = 0.01, T = 10$. }
    \label{tab:high_order}
    \begin{adjustbox}{width=0.35\textwidth}
    \begin{tabular}{|l|l|l|l|l|}
    \hline
         $n_a$ & 4 & 6 & 8 & 10  \\ \hline
         $\gamma$ & 0.6926 & 0.5436 & 0.5167 & 0.5166 \\ \hline
    \end{tabular}
    \end{adjustbox}
\end{table}
\begin{rem}
This behaviour is not seen in the model-based control since any controller order with $n_a \geq 4, n_b \geq 3$ leads to the deadbeat control. However, for data-driven control, the extra order helps to reduce $\gamma$. 
\end{rem}
\section{Extensions}
\label{sec:extensions}




\subsection{Process Noise}

$L_\infty$-bounded process noise $w_t$ with $\norm{w}_\infty \leq \epsilon_w$\ can be added into the \ac{ARX} model \eqref{eq:arx_ground_truth} with
\begin{align}
    y_{t} &= -\textstyle \sum_{i=1}^{n_a} a_i y_{t-i} + \sum_{i=1}^{n_b} b_i u_{t-i} + w_t. \label{eq:arx_ground_truth_w}
\end{align}

With process noise, the equality constraint in \eqref{eq:arx_bilinear_h} is relaxed to a pair of inequalities for each $t=1..T-1$ as
\begin{align}
    -\epsilon_w \leq h_t + \left(-\textstyle \sum_{i=1}^{n_a} a_i \dy_{t-i} + \sum_{i=1}^{n_b} b_i \du_{t-i}\right) - \dy_t  \leq \epsilon_w. \label{eq:arx_bilinear_h_w}
\end{align}

The full program in \eqref{eq:super_full_wsos} may be executed with process noise by posing Psatz constraints over an enlarged consistency set of
\begin{equation}
\label{eq:bpc_w}
    \bpc_w: \left\{\begin{array}{c|c}
    a \in \R^{n_a}, \ b \in \R^{n_b} & \norm{\du}_\infty \leq \epsilon_u  \\
    \du \in \R^{T+n_b-1}  & \norm{\dy}_\infty \leq \epsilon_y \\
    \dy \in \R^{T+n_a} & \textrm{Ineq.  \eqref{eq:arx_bilinear_h_w} holds} \ \forall t = 1..T \\
    \end{array}\right\}
\end{equation}

The Alternatives Psatz in \eqref{eq:altern_sos} may be derived for the set $\bpc_w$ by defining functions $\mu^\pm_t(a, b)$ that are nonnegative over $\pi^{a,b}\bpc_w$ replacing \eqref{eq:Q_const} with
\begin{align}
    Q^w(a,b) &= -q(a,b) + h^T (\mu^+ - \mu^-) + \epsilon_w \1^T(\mu^+ + \mu^-) \nonumber \\
    & + \epsilon_u\1^T(\psi^+ + \psi^-) + \epsilon_y\1^T(\zeta^+ + \zeta^-),
\end{align} 
and every instance of $\mu$ in \eqref{eq:altern_sos_mu}-\eqref{eq:altern_sos_mu} with $\mu^+ - \mu^-$.

\subsection{Switched ARX Models}

The switched setting involves $N_s$ \ac{ARX} models (subsystems) with per-subsystem parameters $(a^k, b^k)$ and lengths $(n_a^k, n_b^k)$ for $k=1..N_s$. A switching sequence $S$  takes on values $S(t) \in 1..N_s$ for each $t \in 1..T$. The ground-truth switched \ac{ARX} model given $S$ is
\begin{align}
    y_{t} &= -\textstyle \sum_{i=1}^{n_a^{S(t)}} a_i^{S(t)} y_{t-i} + \sum_{i=1}^{n_b^{S(t)}} b_i^{S(t)} u_{t-i}.  \label{eq:arx_ground_truth_switched}
\end{align}

Performing a substitution as in \eqref{eq:dydu_corruption} allows for the creation of a consistency set $\bpc^S(\{a_k, b_k\}_{k=1}^{N_s}, \du, \dy)$. The projection $\pi^{a,b} \bpc^S(\{a_k, b_k\}_{k=1}^{N_s})$ has $n_{\textrm{sys}}=\sum_{k=1}^{N_s} n_a^k + n_b^k$ variables. The superstabilization problem for switched \ac{ARX} with a known switching sequence involves finding subsystem-controllers $(\tilde{a}^k, \tilde{b}^k)$ such that each  subsystem-closed-loop transfer functions in \eqref{eq:closed_loop} is superstable. An Alternatives Psatz (modification of \eqref{eq:altern_sos}) would involve Gram matrices of maximal size $\binom{n_{\textrm{sys}}+d}{d}$.
\section{Conclusion}
\label{sec:conclusion}
This paper proposed an algorithm to perform superstabilizing control of \ac{EIV}-corrupted \ac{ARX} systems. The controller $(\tilde{a}, \tilde{b})$ is recovered by solving \iac{POP} using the moment-\ac{SOS} hierarchy of \acp{SDP}. Utilizing the Theorem of Alternatives results in a significantly more tractable program as compared to the Full case.

Future work involves forming superstabilizing controllers for Multi-Input Multi-Output systems using Matrix Fraction Descriptions \citep{chen2005design} and for Linear-Parameter Varying systems. 



\bibliography{references.bib}

\end{document}